\documentclass[12pt]{article}
\usepackage{amsmath}
\usepackage{amssymb}
\usepackage{graphicx}
\usepackage[numbers,sort&compress]{natbib}
\usepackage[colorlinks,
            anchorcolor=blue,
            citecolor=green]
            {hyperref}
\title{Well-posedness for the $\bar\partial$-problem relevant to the AKNS spectral problem}

\author{Junyi Zhu\thanks{E-mail: jyzhu@zzu.edu.cn}, Huan Liu\thanks{E-mail: liuhuan@zzu.edu.cn}\\
\scriptsize{\sl School of Mathematics and Statistics, Zhengzhou University, Zhengzhou, Henan 450001, China}\\
}

\date{}
\begin{document}
\newtheorem{theorem}{Theorem}
\newtheorem{definition}{Definition}
\newtheorem{lemma}{Lemma}
\newtheorem{corollary}{Corollary}
\newtheorem{proposition}{Proposition}
\newtheorem{remark}{Remark}
\hypersetup{CJKbookmarks=true}

\maketitle
\begin{abstract}
The well-posedness for the Dbar problem associated with the AKNS spectral problem is considered. In general, the relevant Dbar equation with normalization condition is equivalent to an integral equation, where the kernel involves exponents $\mathrm{e}^{\pm2ikx}$ with physical variable $x$ as a parameter. We develop a decomposition technique to control the convergence of the integral by defining a new integral operator $RT_{\mathbb{C}}(k;x)$.  The small norm condition of the operator is obtained to show that there exists a unique solution for the Dbar problem. Moreover, the Dbar dressing method is extended to construct the AKNS spectral problem and the potential construction is presented via the Dbar data. Prior estimates are given to show that the map from the Dbar data to the AKNS potential is Lipschitz continuous.


Keywords: well-posedness; Dbar problem; AKNS spectral problem; decomposition technique
\end{abstract}
 \tableofcontents

\section{Introduction}

 The $\bar\partial$ (Dbar) problem can be regarded as a generalization of the Cauchy-Riemann equation for non-analytical functions. Since Beals and Coifman introduced the Dbar problem
  to study the integrable system \cite{B-C1981,B-C1982,ip5-87}, Dbar problem has played important roles to construct explicit solutions \cite{ip4-123,D-L2007,faa19-89,jpa21-L537,KBG1992,KBG1993,Santini2003,sam69-135,jpa46-035204,mpag17-49,aml66-47,wm115-103051,sam148-433,amp13-51,pla499-129359,jgp177-104550,aml107-107297,
 aml117-107143,FEG2022,jmp62-093510,jnmp28-492,aml134-108378,nd111-3689,nd111-5655,nd111-8673,jnmp30-201,aml140-108589,prsla481-20250342,ps100-115237,prsa481-20240764}, and to discuss the asymptotic behaviors of the nonlinear integrable system \cite{imrn2006-48673,jmp63-113504,cmp402-2879,jde329-31,aim426-109088,jde386-214,jde447-113685,qtds24-106,pd456-133915,jde377-121,jmp66-041502,aim484-110552,sam155-e70093}.
 In particular, it is effective in case where the eigenfunction is nowhere analytic, failure for the classical Riemann-Hilbert problem, but the Dbar problem remains valid \cite{sam69-135}.

For integrable system, one needs to consider the following nonhomogeneous Dbar equation
\begin{equation}\label{dba6}
 \bar\partial\psi(k,\bar{k}):=\frac{\partial\psi(k,\bar{k})}{\partial\bar{k}}=\psi(k,\bar{k})R(k,\bar{k}),
 \end{equation}
where $R(k,\bar{k})$ is called the spectral transformation matrix function, or a distribution. Under the normalization conditon $\psi(k,\bar{k})\to I, k\to\infty$, the Dbar problem is equivalent to a integral equation
\begin{equation}\label{dbaa7}
\psi(k,\bar{k})=I+\psi R{T}_{\mathbb{C}}(k),
\end{equation}
where ${T}_{\mathbb{C}}(k)$ acting to the left is called Cauchy-Green operator
\begin{equation}\label{dbaa8}
\psi R{T}_{\mathbb{C}}(k)=\frac{1}{2i\pi}\iint_{\mathbb{C}}\frac{\psi(z,\bar{z})R(z,\bar{z})}{z-k}\mathrm{d}z\wedge \mathrm{d}\bar{z}.
\end{equation}

It is worth noting that a matrix integrable operator
\[\mathcal{K}[\nu](z,\bar{z})=\iint_\mathcal{D}\mathbf{K}(z,\bar{z};w,\bar{w})\nu(w)\frac{\mathrm{d}\bar{w}\wedge \mathrm{d}w}{2i}, \quad z,\bar{z}\in \mathcal{D}\]
\[\mathbf{K}(z,\bar{z};w,\bar{w})=\frac{f^T(z,\bar{z})g(w,\bar{w})}{z-w}, \quad f^T(z,\bar{z})g(z,\bar{z})=0,\]
on a bounded domain $\mathcal{D}$ is considered in \cite{non37-085008}, and the resolvent of the integral operator $Id - \mathcal{K}$ is obtained through the solution of a Dbar-Problem
\[\begin{aligned}
\bar\partial\Gamma(z,z)=\Gamma(z,z)\mathbf{M}(z,\bar{z}); \quad \Gamma\to I, z\to\infty,\\
\mathbf{M}(z,\bar{z})=\pi f(z,\bar{z})g^T(z,\bar{z})\chi_\mathcal{D}(z).
\end{aligned}\]
In the following, we will omit the variable $\bar{k}$ and $\bar{z}$ in functions, just write them as $\psi(k)$ and $R(k)$.

For a certain Lax integrable system with linear spectral problem $\partial_x\phi(x;k)=U(x;k)\phi(x;k)$, the potential matrix $U$ is an important characteristic. For convenience, the variable $x\in\mathbb{R}$ is called a physical variable, and $k\in\mathbb{C}$ is a spectral parameter. We note that the spectral transformation matrix $R(k)$ to the relevant Dbar problem will take the same role as the potential matrix $U$ to character the nonlinear integrable PDE. The so-called relevant Dbar problem means that the spectral transformation matrix $R(k)$ should satisfy an certain evolution equation about the physical variable. In this paper, we will discuss the AKNS spectral problem \cite{sam53-249} or Zakharov-Shabat spectral problem \cite{spj34-62} with
\[U=-ik\sigma_3+Q(x), \quad  \sigma_3=\left(\begin{matrix}
1&0\\
0&-1
\end{matrix}\right),\quad Q=\left(\begin{matrix}
0&u\\
v&0
\end{matrix}\right).\]
The relevant Dbar problem with $R(k;x)$ admitting $\partial_xR=-ik[\sigma_3,R]$, which implies that there are exponents $\mathrm{e}^{\pm2ikx}, x\in\mathbb{R},k\in\mathbb{C}$ in $R$, and they will enter the kernel of the integral \eqref{dbaa7}.
There is one key question: In what space the integral in \eqref{dbaa7} with exponents in its kernel and with the whole $k$ plane as its region of integration is convergent?

We remarked that the properties of the integral \eqref{dbaa8} without exponents and physical parameters have been well studied in \cite{INV1962}. Although the Dbar problem has been introduced to investigate nonlinear integrable PDEs, the convergence of the integral with exponents and physical parameters was rarely considered. In the existing applications of Dbar problem to integrable system, an important prerequisite is the existence of the inverse operator $(\mathcal{I}-RT_\mathbb{C})^{-1}$ \cite{ijmpb4-1003,ip5-87,ip4-123,D-L2007,jpa46-035204,mpag17-49}. To ensure this assumption makes sense, a small operator norm for the operator $RT_\mathbb{C}$ should be valid. However, the exponential factors are not always bounded, for example, see $|\mathrm{e}^{\pm2ikx}|=\mathrm{e}^{\mp2x{\rm Im}k}, x\in\mathbb{R},k\in\mathbb{C}$. One needs to give a decomposition for $k\in\mathbb{C}$ and $x\in\mathbb{R}$ and split the spectral transformation matrix $R(k;x)$.

We give the following assumptions:
\begin{itemize}
\item[i.] The spectral transformation matrix $R(k;x)$ is off-diagonal matrix
\[R(k;x)=\left(\begin{matrix}
0& r_+(k)\mathrm{e}^{-2ikx}\\
r_-(k)\mathrm{e}^{2ikx}&0
\end{matrix}\right).\]
Since $\partial_x{\rm diag}R = 0$, we assume that the diagonal part of the spectral transformation matrix $R(k;x)$ is zero.
\item[ii.] There are no Dirac $\delta$ functions in the spectral transformation matrix $R(k;x)$. The soliton solutions for integrable equation are given by the discrete spectra which are related to the Dirac $\delta$ functions in the spectral transformation matrix $R(k;x)$. In general, $R$ is related to the discrete spectrum and continuous spectrum \cite{ip4-123}. If $R$ only contains the Dirac-$\delta$ function, then integral equation \eqref{dbaa7} can be reduced to a closed linear system. The solvability of the Dbar problem or the equation \eqref{dbaa7} will be simple. However, there will be a logical problem. In fact, the condition of $R$ only containing the Dirac-$\delta$ function implies that the spectral analysis is conducted only with discrete spectrum, which is inconsistent with the inverse scattering transform. Otherwise, the equation \eqref{dbaa7} will contain the integral part given by the associated continuous spectrum. Moreover, the Dirac $\delta$ function is a generalized function, neither H\"older continuous and nor $L^p$ function. For simplicity, we exclude the Dirac $\delta$ functions in $R$.
\end{itemize}

In this paper, we shall develop a decomposition technique to discuss the well-posedness of the Dbar problem relevant to the AKNS spectral problem. To this end, we introduce two nilpotent matrices $w_\pm(k;x)$ admitting $R=w_-+w_+$ and define a new operator $RT_\mathbb{C}$ in the following
\begin{equation}\label{dba10}
\begin{aligned}
 \psi RT_\mathbb{C}(k;x)=&[\psi w_-T_{\mathbb{C}^+}(k;x)+\psi w_+T_{\mathbb{C}^-}(k;x)]\chi_{\{x>0\}}\\
 &+[\psi w_-T_{\mathbb{C}^-}(k;x)+\psi w_+T_{\mathbb{C}^+}(k;x)]\chi_{\{x<0\}},
\end{aligned}
\end{equation}
where $\chi$ is the indicator function of the half lines $x\gtrless0$ and the region of integrations $\mathbb{C}^\pm$ be upper and lower half-planes. We note that the exponential factors in integrals $\psi w_\pm T_{\mathbb{C}^\mp}(k;x)\chi_{\{x>0\}}$ and $\psi w_\pm T_{\mathbb{C}^\pm}(k;x)\chi_{\{x<0\}}$ are bounded.

Now we introduce a function set
\[ C^\alpha(\bar{G})=C(\bar{G})+h^\alpha(\bar{G}), \quad 0<\alpha\leq1, \]
where $G$ is a bounded domain and
\[C(\bar{G})=\{f(k): \sup\limits_{k\in \bar{G}}|f(k)|<\infty\},\]
\[h^\alpha(\bar{G})=\left\{f(k): \sup\limits_{k_1,k_2\in \bar{G}}\frac{|f(k_1)-f(k_2)|}{|k_1-k_2|^\alpha}<\infty, (k_1\neq k_2) \right\}.\]
For $f(k)\in C^\alpha(\bar{G})$, the norm is defined by
\begin{equation}\label{dbb9}
 \|f\|_{C^\alpha(\bar{G})}=\|f\|_{L^\infty(\bar{G})}+\|f\|_{h^\alpha(\bar{G})},
\end{equation}
where
\[\|f\|_{L^\infty(\bar{G})}=\sup\limits_{k\in \bar{G}}|f(k)|, \quad \|f\|_{h^\alpha(\bar{G})}=\sup\limits_{k_1,k_2\in \bar{G}}\frac{|f(k_1)-f(k_2)|}{|k_1-k_2|^\alpha}.\]
Thus $C^\alpha(\bar{G})$ is a Banach space.

Let $f(k)$ be given on the $k\in\mathbb{C}$ plane, and satisfies the conditions
\begin{equation}\label{dba1}
f(k)\in L^p(E_1), \quad f_{(\nu)}(k)=|k|^{-\nu}f(k^{-1})\in L^p(E_1),
\end{equation}
where $E_1=\{k:|k|\leq1\}$ and $\nu$ is a real number. The set of such functions will be denoted by $L^{p,\nu}(\mathbb{C})$. For $f(k)\in L^{p,\nu}(\mathbb{C})$, we define the norm by
\begin{equation}\label{dba2}
\|f\|_{L^{p,\nu}(\mathbb{C})}=\|f\|_{L^p(E_1)}+\|f_{(\nu)}\|_{L^p(E_1)}.
\end{equation}
It is verified that the space $L^{p,\nu}(\mathbb{C})$ is a Banach space.

To discuss the Dbar problem relevant to the AKNS problem, one need to consider matrix functions. For the matrix $\phi(k)$, we let $\|\phi(k)\|$ be a certain matrix norm, and $\|\phi(k)\|_{\mathbf{L}^p}$ be the $L^p$ norm related to the matrix norm. To give the prior estimates for the integrals $\psi w_\pm T_{\mathbb{C}^\mp}(k;x)\chi_{\{x>0\}}$ and $\psi w_\pm T_{\mathbb{C}^\pm}(k;x)\chi_{\{x<0\}}$, we need to further decompose the unit domain and its external region. The unit domain is split into two bounded domains
\begin{equation}\label{dba4}
E_1^\pm=\{k:|k|\leq1, {\rm Im}k\gtrless0\},
\end{equation}
and the counterparts of the split external region are mapped into $E_1^\mp$. So the norm of $\phi\in\mathbf{L}^{p,\nu}(\mathbb{C}^\pm)$ can be defined by two norms on bounded domains in the following form
\begin{equation}\label{dba5}
\|\phi\|_{\mathbf{L}^{p,\nu}(\mathbb{C}^\pm)}=\|\phi\|_{\mathbf{L}^p(E_1^\pm)}+\|\phi_{(\nu)}\|_{\mathbf{L}^p(E_1^\mp)}.
\end{equation}
It is verified that the space $L^{p,\nu}(\mathbb{C})$ is a Banach space. We note that $L^{p,\frac{4}{p}}(\mathbb{C})=L^p(\mathbb{C})$. It is known that $C^\alpha(\bar{G})\subset L^p(\bar{G})$ and $L^p(\bar{G})\subset L^q(\bar{G}), (p>q\geq1)$ are not valid for the infinite domain $G$. However, we can get that $L^{p,\lambda}(\mathbb{C})\subset L^p(\mathbb{C})\subset L^{p,\mu}(\mathbb{C})$ for $p\mu\leq4\leq p\lambda$, and
$C^\alpha(\mathbb{C}) \subset L^{p,\nu}(\mathbb{C}), 0<\alpha<1, 2<p\nu$.

In section 2, we present the properties of the integral on bounded domain. In section 3, we develop a decomposition technique, and define a new integral operator on $k\in\mathbb{C}$ with parameter $x\in\mathbb{R}$ relevant to the AKNS spectral problem. The new integral operator can be decomposed into several operators on $\mathbb{C}^\pm$ and $x\gtrless0$, and further be decomposed into $E_1^\pm$ and $x\gtrless0$. Using the properties on the bounded domain and the boundedness of the exponential factors, we discussed the properties for the new and decomposed operators, and find the exist conditions of the inverse operator $(I-RT_{\mathbb{C}})^{-1}(k;x)$. In section 4, we extend the Dbar dressing method based on the existence of the inverse operator $(I-RT_{\mathbb{C}})^{-1}(k;x)$, and give the AKNS potential reconstruction. Then we show that the maps $L^{q,2}(\mathbb{C})\ni r_\pm(k) \to u(x),v(x)\in L^2(\mathbb{R})$ are Lipschitz continuous. In section 5, we give some conclusions and discussions.

\setcounter{equation}{0}
\section{Preliminary for the integral operator to the Dbar problem on bounded domain}

Let $G$ be a bounded domain with boundary $\Gamma$, and $\bar{G}=G\cup\Gamma$. Suppose the functions $\phi(k)\in C(\bar{G})\cap C^1(G), f(k)\in C(\bar{G})$ admitting the Dbar problem
\begin{equation}\label{dbab1}
\bar\partial \phi(k)=\frac{\partial\phi(k)}{\partial\bar{k}}=f(k),
\end{equation}
then the Pompeiu formula implies that
\begin{equation}\label{dbab2}
\phi(k)=\Phi(k)+f\hat{T}_G(k),
\end{equation}
where
\[\Phi(k)=\frac{1}{2\pi i}\int_{\Gamma}\frac{\phi(z)}{z-k}\mathrm{d}z,\]
\begin{equation}\label{bdab3}
 f\hat{T}_G(k)=\frac{1}{2\pi i}\iint_G\frac{f(z)}{z-k}\mathrm{d}z\wedge \mathrm{d}\bar{z}.
\end{equation}

\begin{theorem}\cite{INV1962}\label{dbth1}
Let $G$ be a bounded domain, and $f(k)\in L^1(\bar{G})$, then the integral
\begin{equation}\label{dbb12}
 f\hat{T}_G(k)=\frac{1}{2\pi i}\iint_G\frac{f(z)}{z-k}\mathrm{d}z\wedge \mathrm{d}\bar{z}
\end{equation}
exists for all points k outside $\bar{G}$, $f\hat{T}_G(k)$ is holomorphic outside $\bar{G}$, and vanish at $k=\infty$.
\end{theorem}

\begin{theorem}\cite{INV1962}\label{dbth2}
Let $G$ be a bounded domain, and $f(k)\in L^1(\bar{G})$, then the function
\[ f\hat{T}_G(k)=\frac{1}{2\pi i}\iint_G\frac{f(z)}{z-k}\mathrm{d}z\wedge \mathrm{d}\bar{z}, \eqno\eqref{dbb12}\]
regarded as a function of point $k\in{G}$, exist almost everywhere, and $f\hat{T}_G(k)\in L^p(\bar{G}^*), 1\leq p<2$, where  $G^*$ is an arbitrary bounded domain of the plane. In addition,
\begin{equation}\label{dbkn9}
\bar\partial(f\hat{T}_G(k))=f(k).
\end{equation}
\end{theorem}

The two theorems tell us that $f\hat{T}_G(k)$ is a function in the $k$ plane even though $f(k)$ is defined on the bounded domain $G$.

\begin{lemma}\label{dble1}
Let $G$ be a bounded domain, and $k_1,k_2\in\mathbb{C}$ are different points. Let
\[I(\mu,\nu)=\iint_G\frac{1}{|z-k_1|^\mu|z-k_2|^\nu}\mathrm{d}\sigma_z, \quad 0<\mu,\nu<2,\]
with $\mathrm{d}\sigma_z=|\frac{1}{2i}\mathrm{d}z\wedge \mathrm{d}\overline{z}|=\mathrm{d}z_R\mathrm{d}z_I, ~ (z=z_R+i z_I)$, then
\begin{equation}\label{dbb1}
 I(\mu,\nu)\leq\left\{\begin{array}{ll}
 M_1(\mu,\nu,G), &\mu+\nu<2;\\
 M_2(\mu,\nu,G)+8\pi|\ln|k_1-k_2||, &\mu+\nu=2;\\
 M_3(\mu,\nu)|k_1-k_2|^{2-\mu-\nu}, &\mu+\nu>2,
 \end{array}\right.
\end{equation}
where $ M_1(\mu,\nu,G), M_2(\mu,\nu,G)$ are positive constants depending on $\mu,\nu$ and $G$, and $ M_3(\mu,\nu)$ is a positive constant on $\mu,\nu$.
\end{lemma}
{\bf Proof}. Taking $k_1$ as the center, draw a circle $G_1$ with radius $\beta=2|k_1-k_2|$, and then draw a concentric circle $G_0$ with radius $2\beta_0$, such that $G\subset G_0$. In the annulus $G_0-G_1$, we have
\[|z-k_1|\leq|z-k_2|+|k_2-k_1|\leq2|z-k_2|.\]
We note that there are three cases on the position about the points $k_1,k_2$ and the domain $G$:
\begin{enumerate}
\item[(i).] $k_1,k_2$ locate outside of the domain $G$;
\item[(ii).] $k_1,k_2$ locate in the domain $G$;
\item[(iii).] $k_1$ outside of $G$, and $k_2$ in $G$.
\end{enumerate}
\begin{figure}[htp]
  \centering
  \includegraphics[width=4cm]{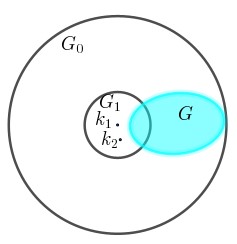}\quad
  \includegraphics[width=4cm]{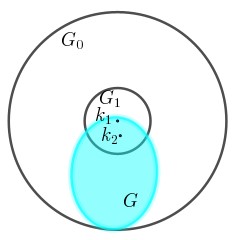}\quad
  \includegraphics[width=4cm]{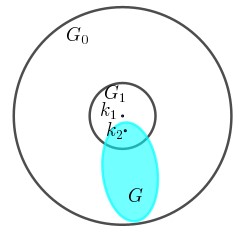}
  \caption{Case i. (left); case ii.(middle);case iii. (right).}\label{Fig1}
\end{figure}

We consider the integral
\begin{equation}\label{dbb2}
\begin{aligned}
I_0&:=\iint_{(G_0-G_1)\cap G}\frac{1}{|z-k_1|^\mu|z-k_2|^\nu}\mathrm{d}\sigma_z\\
&\leq\iint_{(G_0-G_1)\cap G}\frac{1}{|z-k_1|^\mu(2^{-1}|z-k_1|)^\nu}\mathrm{d}\sigma_z\\
&\leq2^\nu(\theta_1-\theta_0)\int_\beta^{2\beta_0}r^{1-\mu-\nu}\mathrm{d}r,
\end{aligned}
\end{equation}
where $\theta_1 - \theta_0$ is the central angle (measured counterclockwise) subtended by the region $(G_0 - G_1) \cap G$ with respect to $k_1$. It is noted that $0<\theta_1-\theta_0<2\pi$ for the cases (i) and (iii), and  $0\leq\theta_1-\theta_0\leq2\pi$ for the case (ii).

If $\mu+\nu<2$, then \eqref{dbb2} implies that
\begin{equation}\label{dba7}
I_0\leq \frac{\pi 2^{1+\nu}}{2-\mu-\nu}((2\beta_0)^{2-\mu-\nu}-\beta^{2-\mu-\nu})\leq \frac{\pi 2^{3-\mu}}{2-\mu-\nu}\beta_0^{2-\mu-\nu}.
\end{equation}
If $\mu+\nu=2$, then we have from \eqref{dbb2} that
\begin{equation}\label{dba8}
I_0\leq\pi2^{1+\nu}\ln r|_\beta^{2\beta_0}\leq\pi2^{1+\nu}(|\ln\beta_0|+|\ln|k_1-k_2||).
\end{equation}
And for $\mu+\nu>2$, we have
\begin{equation}\label{dba9}
I_0\leq\frac{\pi2^{1+\nu}}{\mu+\nu-2}(2|k_1-k_2|)^{2-\mu-\nu}=\frac{\pi2^{3-\mu}}{\mu+\nu-2}|k_1-k_2|^{2-\mu-\nu}.
\end{equation}

Next, we consider the integral
\[I_1:=\iint_{G_1\cap G}\frac{1}{|z-k_1|^\mu|z-k_2|^\nu}\mathrm{d}\sigma_z,\]
and let
\[d_0=\inf\left\{\frac{|k_1-k|}{|k_1-k_2|}:\forall k\in G_1\cap G\right\}.\]
Thus $0<d_0<2$ for the cases (i) and (iii), $0\leq d_0\leq2$ for the case (ii).

For $z\in G_1\cap G$, we let
\[z-k_1=\omega(k_1-k_2), \quad |\omega|=w, \quad r=|z-k_1|=w|k_1-k_2|,\]
then $d_0\leq w\leq2$, and
\[z-k_2=(\omega+1)(k_1-k_2), \quad |z-k_2|=|\omega+1||k_1-k_2|,\]
\[\mathrm{d}\sigma_z=r\mathrm{d}r\mathrm{d}\theta=|k_1-k_2|^2w\mathrm{d}w\mathrm{d}\theta=|k_1-k_2|^2\mathrm{d}\sigma_\omega.\]
Thus we find that
\begin{equation}\label{dbb3}
\begin{aligned}
I_1&\leq |k_1-k_2|^{2-\mu-\nu}\iint_{d_0\leq w\leq2}w^{-\mu}|\omega+1|^{-\nu}\mathrm{d}\sigma_\omega\\
&\leq M(\mu,\nu)|k_1-k_2|^{2-\mu-\nu}.
\end{aligned}
\end{equation}
Here the integral
\[\begin{aligned}
&\big|\iint_{0\leq w\leq2}w^{-\mu}|\omega+1|^{-\nu}\mathrm{d}\sigma_\omega\big|\\
&\leq\big|\iint_{0\leq w\leq1}w^{-\mu}|\omega+1|^{-\nu}\mathrm{d}\sigma_\omega\big|+
\big|\iint_{1\leq w\leq2}w^{-\mu}|\omega+1|^{-\nu}\mathrm{d}\sigma_\omega\big|\\
&\leq2\pi\int_0^1w^{1-\mu}(1-w)^{-\nu}\mathrm{d}w+\big|\iint_{1\leq w\leq2}w^{-\mu}|\omega+1|^{-\nu}\mathrm{d}\sigma_\omega\big|,
\end{aligned}\]
and the right integral is bounded for $0<\mu<2, 0<\nu<2$ by virtue of the theories of the Euler integral and the hypergeometric function.

Since $I(\mu,\nu)=I_0+I_1$, with
\[\begin{aligned}
M_1(\mu,\nu,G)&=M(\mu,\nu)|k_1-k_2|^{2-\mu-\nu}+\frac{\pi 2^{3-\mu}}{2-\mu-\nu}\beta_0^{2-\mu-\nu};\\
M_2(\mu,\nu,G)&=M(\mu,\nu)|k_1-k_2|^{2-\mu-\nu}+\pi2^{1+\nu}|\ln\beta_0|;\\
M_3(\mu,\nu)&=M(\mu,\nu)+\frac{\pi2^{3-\mu}}{\mu+\nu-2}.
\end{aligned}\]
then the Lemma is proved. \qquad $\square$

\begin{theorem}\label{dbth3}
Let $G$ be a bounded domain and $f(k)\in L^p(\bar{G}), p>2$, then the function $g(k)=fT_G(k)$ satisfying the conditions
\begin{equation}\label{dbb4}
|g(k)|\leq M_4(p,G)\|f\|_{L^p(\bar{G})}, \quad k\in\mathbb{C},
\end{equation}
\begin{equation}\label{dbb5}
 |g(k_1)-g(k_2)|\leq M_5(p)\|f\|_{L^p(\bar{G})}|k_1-k_2|^\gamma,
\end{equation}
where $\gamma=\frac{p-2}{p}, k_1,k_2\in\mathbb{C}$. Here the positive constant $M_4(p,G)$ depends on $p$ and $G$, the positive constant $M_5(p)$ depends on $p$.
\end{theorem}
{\bf Proof}. To prove the inequality \eqref{dbb4}, making use of the H\"older inequality, we find that
\begin{equation}\label{dbb6}
\begin{aligned}
|g(k)|&\leq\frac{1}{\pi}\iint_G\frac{|f(z)|}{|z-k|}d\sigma_z\\
&\leq\frac{1}{\pi}\left(\iint_G|f(z)|^pd\sigma_z\right)^{\frac{1}{p}}\left(\iint_G|z-k|^{-q}d\sigma_z\right)^{\frac{1}{q}},
\end{aligned}
\end{equation}
where $q=\frac{p}{p-1}$. For $k\in\bar{G}$, we let $z-k=re^{i\theta}$, and $d$ be the redius of the domain $G$, then
\begin{equation}\label{dbb7}
\begin{aligned}
\left(\iint_G|z-k|^{-q}\mathrm{d}\sigma_z\right)^{\frac{1}{q}}&\leq\left(\int_0^{2\pi}\int_0^dr^{1-q}\mathrm{d}r\mathrm{d}\theta\right)^{\frac{1}{q}}\\
&=\big(\frac{2\pi}{2-q}\big)^{\frac{1}{q}}d^{\frac{2-q}{q}}=(\frac{2\pi}{2-q})^{\frac{1}{q}}d^\gamma,
\end{aligned}
\end{equation}
where the identity $\gamma=\frac{p-2}{p}=\frac{2-q}{q}$ has been used.
For $k\notin\bar{G}$, letting $d_0$ be the distance from $k$ to the domain $G$, then we have
\begin{equation}\label{dbb8}
\begin{aligned}
\left(\iint_G|z-k|^{-q}\mathrm{d}\sigma_z\right)^{\frac{1}{q}}&\leq\left(\int_0^{2\pi}\int_{d_0}^{d_0+d}r^{1-q}\mathrm{d}r\mathrm{d}\theta\right)^{\frac{1}{q}}\\
&=\left[\frac{2\pi}{2-q}\big((d_0+d)^{2-q}-d_0^{2-q}\big)\right]^{\frac{1}{q}}\\
&\leq\left[\frac{2\pi}{2-q}\bigg((1+\frac{d_0}{d})^{2-q}-(\frac{d_0}{d})^{2-q}\bigg)\right]^{\frac{1}{q}} d^\gamma.
\end{aligned}
\end{equation}
Since $2-q>0$ and
\[(1+\frac{d_0}{d})^{2-q}-(\frac{d_0}{d})^{2-q}\to\left\{\begin{array}{cc}
1&d_0\to0,\\
0&d_0\to\infty,
\end{array}\right.\]
then the left side of \eqref{dbb8} is bounded.
Thus the equality \eqref{dbb4} is obtained from \eqref{dbb6},\eqref{dbb7} and \eqref{dbb8}.

Next we prove the condition \eqref{dbb5}. For $p>2, q=\frac{p}{p-1}$, then $1<q<2$. By virtue of the H\"older inequality and the Lemma \ref{dble1}, we obtain that
\[\begin{aligned}
|g(k_1)-g(k_2)|&\leq\frac{|k_1-k_2|}{\pi}\iint_G\frac{|f(z)|}{|z-k_1||z-k_2|}\mathrm{d}\sigma_z\\
&\leq\frac{|k_1-k_2|}{\pi}\|f\|_{L^p(\bar{G})}\left(\iint_G(|z-k_1||z-k_2|)^{-q}\mathrm{d}\sigma_z\right)^{\frac{1}{q}}\\
&\leq\frac{|k_1-k_2|}{\pi}\|f\|_{L^p(\bar{G})}\left(M_3(q,q)|k_1-k_2|^{2-2q}\right)^{\frac{1}{q}}\\
&\leq M_5(p)\|f\|_{L^p(\bar{G})}|k_1-k_2|^\gamma.
\end{aligned}\]
The theorem is proved. \qquad $\square$

It is remarked that, from the conditions \eqref{dbb4} and \eqref{dbb5}, $\hat{T}_G$ is a linear completely continuous operator in the space $L^p(\bar{G})$, and
\[\hat{T}_G: L^p(\bar{G})\to C^\gamma(\mathbb{C}), \quad 0<\gamma<1, \gamma=\frac{p-2}{p}, p>2.\]
Here
\[C^\gamma(\mathbb{C})=\{f(k): f(k),f(k^{-1})\in C^\gamma(E_1)\}.\]

\setcounter{equation}{0}
\section{Dbar problem relevant to the AKNS spectral problem}
In this section, we will develop a decomposition technique for the Dbar problem and extend the above properties to consider the AKNS spectral problem. For convenience, the $k$ plane will be called spectral space, and $x,t$ be physical variables. For Lax integrable system with linear spectral problem $\phi_x=U\phi$, the potential matrix $U=U(x;k)$ can be regarded as an important characteristic to certain nonlinear integrable PDE.

It is noted that the Dbar problem
\begin{equation}\label{dbc1}
\bar\partial\psi(k)=\psi(k)R(k)
\end{equation}
can be applied to investigate the integrable system, the spectral transformation matrix or the distribution $R(k)$ will take the same role as the potential matrix $U$ to character the certain nonlinear integrable PDE. It is known that the Dbar equation \eqref{dbc1} with normalization condition $\psi(k)\to I$ as $k \to\infty$ is equivalent to the integral equation
\begin{equation}\label{dbc2}
 \psi(k)=I+\psi R{T}_\mathbb{C}(k),
\end{equation}
where the operator ${T}_\mathbb{C}$ acting to the left is defined in the following form
\begin{equation}\label{dbc3}
 \psi R{T}_\mathbb{C}(k)=\frac{1}{2i\pi}\iint_{\mathbb{C}}\frac{\psi(z)R(z)}{z-k}\mathrm{d}z\wedge \mathrm{d}\bar{z}.
 \end{equation}
Here an assumption should be given that the integral is convergent for certain $R$ and $\psi$.

To discuss the AKNS spectral problem by the Dbar problem, one need to introduce the physical variable $x$ into $R(k)$ as following
\begin{equation}\label{gaa5}
\partial_x R=-ik[\sigma_3,R], \quad \sigma_3={\rm diag}(1,-1),
\end{equation}
which implies that
\begin{equation}\label{gaa6}
R(k;x)=\mathrm{e}^{-ikx\hat\sigma_3}r(k)=\left(\begin{matrix}
0&r_+(k)\mathrm{e}^{-2ikx}\\
r_-(k)\mathrm{e}^{2ikx}&0
\end{matrix}\right),
\end{equation}
where the off-matrix function $r(k)$ or $r_\pm(k)$ are independent of $x$, $[\sigma_3,A]=\sigma_3A-A\sigma_3$ and $\mathrm{e}^{\alpha\hat\sigma_3}A=\mathrm{e}^{\alpha\sigma_3}A\mathrm{e}^{-\alpha\sigma_3}$ for some $2\times2$ matrix $A$.
Thus the associated eigenvalue function $\psi$ will dependent on the physical variable $x$, that is $\psi=\psi(k;x)$. It is emphasized that, substituting $R(k;x)$ into the integral equation \eqref{dbc2} directly, it may lead to uncertainty.
Since the exponentials $\mathrm{e}^{\pm 2ikx}$ with $x\in\mathbb{R}$ and $k\in\mathbb{C}$ are not bounded for $x{\rm Im}k\lessgtr0$,
the function $\psi(\cdot,x)R(\cdot,x)\hat{T}_\mathbb{C}(k)$ may not convergent in view of the integral \eqref{dbc2}.

To control the spectral transformation matrix $R(k;x)$ in \eqref{gaa6}, we split it into two nilpotent matrices $w_\pm(k;x)$, that is $R=w_-+w_+$, where
\begin{equation}\label{dbc4}
w_-=\left(\begin{matrix}
0&0\\
r_-(k)\mathrm{e}^{2ikx}&0
\end{matrix}\right), \quad w_+=\left(\begin{matrix}
0&r_+(k)\mathrm{e}^{-2ikx}\\
0&0
\end{matrix}\right).
\end{equation}
Then we define a new integral operator $RT_\mathbb{C}(k;x)$ by giving the following decomposition  \begin{equation}\label{dbc5}
\begin{aligned}
 \psi RT_\mathbb{C}(k;x):=&[\psi w_-T_{\mathbb{C}^+}(k;x)+\psi w_+T_{\mathbb{C}^-}(k;x)]\chi_{\{x>0\}}\\
 &+[\psi w_-T_{\mathbb{C}^-}(k;x)+\psi w_+T_{\mathbb{C}^+}(k;x)]\chi_{\{x<0\}},
\end{aligned}
\end{equation}
where $\chi$ is the indicator function of the half lines $\{x\gtrless0\}$. Here
\begin{equation}\label{dbknb8}
\begin{aligned}
\psi w^-T_{\mathbb{C}^\pm}(k;x)=\frac{1}{2\pi i}\iint_{\mathbb{C}^\pm}\frac{\psi(z;x)w^-(z;x)}{z-k}\mathrm{d}z\wedge \mathrm{d}\bar{z},\\
\psi w^+T_{\mathbb{C}^\pm}(k;x)=\frac{1}{2\pi i}\iint_{\mathbb{C}^\pm}\frac{\psi(z;x)w^+(z;x)}{z-k}\mathrm{d}z\wedge \mathrm{d}\bar{z}.
\end{aligned}
\end{equation}

\begin{lemma}\label{dble2}
Let $\psi\in {L}_x^\infty(\mathbb{R}^+,\mathbf{L}_k^{p,0}(\mathbb{C}^\pm))$, $r_\mp \in L_k^{q,2}(\mathbb{C}^\pm)$, and
introduce matrix functions
\begin{equation}\label{dbc6}
g^\pm(k;x)=\psi w_\mp T_{\mathbb{C}^\pm}(k;x)\chi_{\{x>0\}},
\end{equation}
then we have the following conditions
\begin{equation}\label{dbc7}
\|g^\pm(k;x)\|\leq M_{\pm}^{p,q}\|r_\mp\|_{L^{q,2}(\mathbb{C}^\pm)}\sup\limits_{x>0}\|\psi\|_{\mathbf{L}_k^{p,0}(\mathbb{C}^\pm)},
\end{equation}
\begin{equation}\label{dbc8}
 \|g^{\pm}(k_1)-g^{\pm}(k_2)\|\leq \tilde{M}_\pm^{p,q}\|r_\mp\|_{L^{q,2}(\mathbb{C}^\pm)}\sup\limits_{x>0}\|\psi\|_{\mathbf{L}_k^{p,0}(\mathbb{C}^\pm)}\|k_1-k_2|^\alpha,
\end{equation}
where $k_1,k_2\in\mathbb{C}$, $\frac{1}{p}+\frac{1}{q}<\frac{1}{2}$, and $\alpha=1-2\big(\frac{1}{p}+\frac{1}{q}\big)$. 
Here the positive constants $M_{\pm}^{p,q},\tilde{M}_\pm^{p,q}$ depend on $p$ and $q$.
\end{lemma}
{\bf Proof}. For convenience, we introduce the following domains
\[ E_2^\pm=\{k:|k|\geq1, {\rm Im}k\gtrless0\}.\]
We split the infinite domain $\mathbb{C}^\pm$ into $E_1^\pm$ and $E_2^\pm$, and let
\begin{equation}\label{dbc9}
 g^\pm(k;x)=g_1^\pm(k;x)+g_2^\pm(k;x),
\end{equation}
where
\begin{equation}\label{dbc10}
\begin{aligned}
g_1^\pm(k;x)=\psi w_\mp T_{E_1^\pm}(k;x)\chi_{\{x>0\}},\\
g_2^\pm(k;x)=\psi w_\mp T_{E_2^\pm}(k;x)\chi_{\{x>0\}}.
\end{aligned}
\end{equation}

Since $E_1^\pm$ is a bounded domain, using the H\"older inequality and the Theorem \ref{dbth3}, we have
\begin{equation}\label{dbc11}
\begin{aligned}
\|g_1^\pm(k;x)\|&\leq M_{4,\pm}^\mu\sup\limits_{x>0}\|\psi w_\mp\|_{\mathbf{L}_k^\mu(E_1^\pm)}\\
&\leq M_{4,\pm}^{p,q}\|r_\mp(k)\|_{L^q(E_1^\pm)}\sup\limits_{x>0}\|\psi\|_{\mathbf{L}_k^p(E_1^\pm)},
\end{aligned}
\end{equation}
where $\mu>2$ and $\frac{1}{p}+\frac{1}{q}=\frac{1}{\mu}$.

For the infinite domains $E_2^\pm$, they can be mapped into $E_1^\mp$ via the transformation $k\to k^{-1}$. Thus we find that
\begin{equation}\label{dbc12}
\begin{aligned}
g_2^\pm(k;x)&=\frac{1}{2i\pi}\iint_{E_2^\pm}\frac{\psi(z;x)w_\mp(z;x)}{z-k}\mathrm{d}z\wedge \mathrm{d}\bar{z} ~\chi_{\{x>0\}}\\
&=\frac{1}{2i\pi}\iint_{E_1^\mp}\frac{\psi(z^{-1};x)w_\mp(z^{-1};x)}{\bar{z}^{2}}(\frac{1}{z}-\frac{1}{z-k^{-1}})\mathrm{d}z\wedge \mathrm{d}\bar{z}~ \chi_{\{x>0\}}\\
&:=g_0^\mp(0;x)-g_0^\mp(k^{-1};x),
\end{aligned}
\end{equation}
where
\[
\begin{aligned}
 g_0^\mp(\lambda;x)&=\frac{1}{2i\pi}\iint_{E_1^\mp}\frac{\psi(z^{-1};x)\bar{z}^{-2}w_\mp(z^{-1};x)}{z-\lambda}\mathrm{d}z\wedge \mathrm{d}\bar{z}~\chi_{\{x>0\}}\\
 &=\frac{1}{2i\pi}\iint_{E_1^\mp}\frac{\psi_{(0)}(z;x)w_{(2),\mp}(z;x)}{z-\lambda}\mathrm{d}z\wedge \mathrm{d}\bar{z}~\chi_{\{x>0\}}.
\end{aligned}\]

Here we have use the notation in \eqref{dba1}. It is noted that the exponents $\mathrm{e}^{\pm2iz^{-1}x}$ in $w_{(2),\mp}(z;x)$ are bounded in $E_1^\mp$ for $x>0$. Applying same discussion to $g_0^\mp(\lambda;x)$ as $g_1^\pm(k;x)$ in \eqref{dbc11}, and using the triangle inequality, we have
\begin{equation}\label{dbc13}
\|g_2^\pm(k;x)\|\leq \tilde{M}_{4,\pm}^{p,q}\|r_{(2),\mp}\|_{L_k^q(E_1^\mp)}\sup\limits_{x>0}\|\psi_{(0)}\|_{\mathbf{L}_k^p(E_1^\mp)}.
\end{equation}
Then using the triangle inequality again, the condition \eqref{dbc7} is obtained by the substitution of estimates \eqref{dbc11} and \eqref{dbc13} into \eqref{dbc9}.

Next, we prove the condition \eqref{dbc8}. For $g_1^\pm(k;x)$ and $k_1,k_2\in\mathbb{C}$, using the H\"older inequality and the Theorem \ref{dbth3}, we have
\begin{equation}\label{dbc14}
\begin{aligned}
\|g_1^\pm(k_1;x)-g_1^\pm(k_2;x)\|&\leq M_{5,\pm}^{\mu}\sup\limits_{x>0}\|\psi w_\mp\|_{\mathbf{L}_k^\mu(E_1^\pm)}\|k_1-k_2|^{\frac{\mu-2}{\mu}}\\
&\leq M_{5,\pm}^{p,q}\|r_\pm\|_{L_k^q(E_1^\pm)}\sup\limits_{x>0}\|\psi\|_{\mathbf{L}_k^p(E_1^\pm)}\|k_1-k_2|^\alpha,
\end{aligned}
\end{equation}
where $\mu>2$, $\frac{1}{p}+\frac{1}{q}=\frac{1}{\mu}$ and $\alpha=\frac{\mu-2}{\mu}$.

Now, we give the same estimation for $g_2^\pm(k;x)$ given in \eqref{dbc12}. To this end, we find that
\begin{equation}\label{dbc15}
\begin{aligned}
&\|g_2^\pm(k_1;x)-\|g_2^\pm(k_2;x)\|\\
&=\|g_0^\mp(k_1^{-1};x)-\|g_0^\mp(k_2^{-1};x)\|\\
&\leq\frac{|k_1-k_2|}{\pi}\iint_{E_1^\mp}\|\psi_{(0)}(z;x)\||r_{(2),\mp}(z)|\frac{1}{|1-zk_1||1-zk_2|}\mathrm{d}\sigma_z.
\end{aligned}
\end{equation}
We will consider the following cases about the kernel in \eqref{dbc15}:

(1).  If $|k_j|\leq\frac{1}{2}, j=1,2$, then for $|z|\leq1$, $|zk_j|\leq\frac{1}{2}$ and $\frac{1}{|1-zk_j|}\leq2$. In this case, we can get
\begin{equation}\label{dbc16}
\begin{aligned}
&\|g_2^\pm(k_1;x)-g_2^\pm(k_2;x)\|\\
&\leq\frac{|k_1-k_2|}{\pi}\tilde{M}_{5,\mp}^\mu\sup\limits_{x>0}\|\|\psi_{(0)}(z;x)\||r_{(2),\mp}(z)|\|_{{L}_k^\mu(E_1^\mp)}\\
&\leq|k_1-k_2|\tilde{M}_{5,\mp}^{p,q}\sup\limits_{x>0}\|\psi_{(0)}\|_{\mathbf{L}_k^{p}(E_1^\mp)}\|r_{(2),\mp}\|_{{L}_k^{q}(E_1^\mp)}\\
&\leq\tilde{M}_{5,\mp}^{p,q}\sup\limits_{x>0}\|\psi_{(0)}\|_{\mathbf{L}_k^{p}(E_1^\mp)}\|r_{(2),\mp}\|_{{L}_k^{q}(E_1^\mp)}|k_1-k_2|^\alpha,
\end{aligned}
\end{equation}
in view of $|k_1-k_2|<1$.

(2). If $|k_1|<\frac{1}{2}, |k_2|\geq\frac{1}{2}$, then $\frac{1}{|1-zk_1|}\leq2$. Thus we have
\[\begin{aligned}
&\|g_2^\pm(k_1;x)-g_2^\pm(k_2;x)\|\\
&\leq\frac{|k_1-k_2|}{k_2\pi}\tilde{M}_{6,\mp}^\mu\sup\limits_{x>0}\|\|\psi_{(0)}(z;x)\||r_{(2),\mp}(z)|\|_{{L}_z^\mu(E_1^\mp)}
\big\|\frac{1}{z-k_2^{-1}}\big\|_{{L}_z^\nu(E_1^\mp)},
\end{aligned}\]
where $\nu=\frac{\mu-2}{\mu}$. Since $|k_2^{-1}|\leq2$, then $|z-k_2^{-1}|<|z|+|k_2^{-1}|<3$ and
\begin{equation}\label{dbc19a}
\big\|\frac{1}{z-k_2^{-1}}\big\|_{{L}_k^\nu(E_1^\mp)}\leq\left(\pi\int_0^3r^{1-\nu}\mathrm{d}r\right)^{\frac{1}{\nu}}
=\left(\frac{\pi3^{2-\nu}}{2-\nu}\right)^{\frac{1}{\nu}}.
\end{equation}
In addition, using $|\frac{k_1}{k_2}-1|<2$, we can get that
\[\begin{aligned}
\frac{|k_1-k_2|}{k_2}&=\frac{|k_1-k_2|^{\frac{\mu-2}{\mu}}|k_1-k_2|^{\frac{2}{\mu}}}{|k_2|^{\frac{\mu-2}{\mu}}|k_2|^{\frac{2}{\mu}}}
=\frac{|\frac{k_1}{k_2}-1|^{\frac{2}{\mu}}}{|k_2|^{\frac{\mu-2}{\mu}}} |k_1-k_2|^{\frac{\mu-2}{\mu}}\\
&\leq2|k_1-k_2|^\alpha.
\end{aligned}
\]
Thus we have
\begin{equation}\label{dbc17}
\|g_2^\pm(k_1;x)-g_2^\pm(k_2;x)\|\leq\tilde{M}_{7,\mp}^{p,q}\sup\limits_{x>0}\|\psi_{(0)}\|_{\mathbf{L}_k^{p}(E_1^\mp)}\|r_{(2),\mp}\|_{{L}_k^{q}(E_1^\mp)}|k_1-k_2|^\alpha.
\end{equation}

(3). If $|k_j|\geq\frac{1}{2}, j=1,2$, then using \eqref{dbc19a}, we find that
\begin{equation}\label{dbc18}
\begin{aligned}
&\|g_2^\pm(k_1;x)-g_2^\pm(k_2;x)\|=\|g_0^\mp(k_1^{-1};x)-\|g_0^\mp(k_2^{-1};x)\|\\
&\leq\tilde{M}_{8,\mp}^{p,q}\sup\limits_{x>0}\|\psi_{(0)}\|_{\mathbf{L}_k^{p}(E_1^\mp)}\|r_{(2),\mp}\|_{{L}_k^{q}(E_1^\mp)}|\frac{1}{k_1}-\frac{1}{k_2}|^\alpha\\
&\leq\tilde{M}_{9,\mp}^{p,q}\sup\limits_{x>0}\|\psi_{(0)}\|_{\mathbf{L}_k^{p}(E_1^\mp)}\|r_{(2),\mp}\|_{{L}_k^{q}(E_1^\mp)}|k_1-k_2|^\alpha.
\end{aligned}
\end{equation}

From \eqref{dbc14}-\eqref{dbc18}, we can get that
\[\begin{aligned}
&\|g^{\pm}(k_1)-g^{\pm}(k_2)\|\\
&\leq \tilde{M}_{\pm}^{p,q}|k_1-k_2|^\alpha\left(\sup\limits_{x>0}\|\psi\|_{\mathbf{L}_k^p(E_1^\pm)}\|r_\pm\|_{L_k^q(E_1^\pm)}+\sup\limits_{x>0}\|\psi_{(0)}\|_{\mathbf{L}_k^{p}(E_1^\mp)}\|r_{(2),\mp}\|_{{L}_k^{q}(E_1^\mp)}\right)\\
&\leq \tilde{M}_{\pm}^{p,q}|k_1-k_2|^\alpha\left(\sup\limits_{x>0}\|\psi\|_{\mathbf{L}_k^p(E_1^\pm)}+\sup\limits_{x>0}\|\psi_{(0)}\|_{\mathbf{L}_k^{p}(E_1^\mp)}\right)\left(\|r_\pm\|_{L_k^q(E_1^\pm)}+\|r_{(2),\mp}\|_{{L}_k^{q}(E_1^\mp)}\right).
\end{aligned}\]
Thus the condition \eqref{dbc8} is proved in view of the definition \eqref{dba5}. \qquad $\square$

\begin{lemma}\label{dble3}
Let $\psi\in {L}_x^\infty(\mathbb{R}^-,\mathbf{L}_k^{p,0}(\mathbb{C}^\pm))$, $r_\pm \in L_k^{q,2}(\mathbb{C}^\pm)$, and
 let matrix functions
\begin{equation}\label{dbc23a}
h^\pm(k;x)=\psi w_\pm T_{\mathbb{C}^\pm}(k;x)\chi_{\{x<0\}},
\end{equation}
then we have the following conditions
\begin{equation}\label{dbc19}
\|h^\pm(k;x)\|\leq N_{\pm}^{p,q}\|r_\pm\|_{L^{q,2}(\mathbb{C}^\pm)}\sup\limits_{x<0}\|\psi\|_{\mathbf{L}_k^{p,0}(\mathbb{C}^\pm)},
\end{equation}
\begin{equation}\label{dbc20}
 \|h^{\pm}(k_1;x)-h^{\pm}(k_2;x)\|\leq \tilde{N}_\pm^{p,q}\|r_\pm\|_{L^{q,2}(\mathbb{C}^\pm)}\sup\limits_{x<0}\|\psi\|_{\mathbf{L}_k^{p,0}(\mathbb{C}^\pm)}\|k_1-k_2|^\alpha,
\end{equation}
where $\frac{1}{p}+\frac{1}{q}<\frac{1}{2}$, and $\alpha=1-2\big(\frac{1}{p}+\frac{1}{q}\big)$. Here the positive constants $N_{\pm}^{p,q},\tilde{N}_\pm^{p,q}$ depend on $p$ and $q$.
\end{lemma}
{\bf Proof}. The lemma can be proved similarly. \qquad $\square$

\begin{theorem}\label{dbth4}
Let $\psi\in {L}_x^\infty(\mathbb{R},\mathbf{L}_k^{p,0}(\mathbb{C}))$, $r_\pm \in L_k^{q,2}(\mathbb{C})$,
then the integral $\psi RT_\mathbb{C}(k;x)$ defined in \eqref{dbc5} satisfies the following conditions
\begin{equation}\label{dbc21}
\|\psi RT_\mathbb{C}(k;x)\|\leq M^{p,q}(\|r_+\|_{L^{q,2}(\mathbb{C})}+\|r_-\|_{L^{q,2}(\mathbb{C})})\sup\limits_{x\in\mathbb{R}}\|\psi\|_{\mathbf{L}_k^{p,0}(\mathbb{C})},
\end{equation}
and for $k_1,k_2\in\mathbb{C}$,
\begin{equation}\label{dbc22}
\begin{aligned}
&\|\psi RT_\mathbb{C}(k_1;x)-\psi RT_\mathbb{C}(k_2;x)\|\\
&\leq \tilde{M}^{p,q}(\|r_+\|_{L^{q,2}(\mathbb{C})}+\|r_-\|_{L^{q,2}(\mathbb{C})})\sup\limits_{x\in\mathbb{R}}\|\psi\|_{\mathbf{L}_k^{p,0}(\mathbb{C})}\|k_1-k_2|^\alpha,
\end{aligned}
\end{equation}
where $\frac{1}{p}+\frac{1}{q}<\frac{1}{2}$, and $\alpha=1-2\big(\frac{1}{p}+\frac{1}{q}\big)$.
Here the positive constants $M^{p,q},\tilde{M}^{p,q}$ depend on $p$ and $q$.
\end{theorem}
{\bf Proof}. We note that $\|r_\pm\|_{L^{q,2}(\mathbb{C}^\pm)}\leq\|r_\pm\|_{L^{q,2}(\mathbb{C})}$ and
$\sup\limits_{x\gtrless0}\|\psi\|_{\mathbf{L}_k^{p,0}(\mathbb{C}^\pm)}\leq\sup\limits_{x\in\mathbb{R}}\|\psi\|_{\mathbf{L}_k^{p,0}(\mathbb{C})}$.
In addition,
\[\psi RT_\mathbb{C}(k;x)=g^+(k;x)+g^-(k;x)+h^+(k;x)+h^-(k;x),\]
and collect the conditions in Lemma \ref{dble2} and Lemma \ref{dble3}, 
the theorem can be proved. \qquad $\square$

The theorem \ref{dbth4} implies that if $r_\pm \in L_k^{q,2}(\mathbb{C})$ then
\[RT_\mathbb{C}: \mathbf{L}_k^{p,0}(\mathbb{C})\to \mathbf{C}_k^\alpha(\mathbb{C}).\]
We note that $C^\alpha(\mathbb{C})\subset L^{p,0}(\mathbb{C})$. In fact, it is easy to see that $C^\alpha(E_1)\subset L^{p}(E_1)$ \cite{INV1962}, and the space $C^\alpha(E_1)$ is density in $L^{p}(E_1)$. Now, outside the unit circle $E_1$, for every $f(k)\in C^\alpha(\mathbb{C}), |k|\geq 1$, that is $f_{(0)}(k)=f(k^{-1})\in C^\alpha(E_1)$, then for $k_1\in E_1$
\[\begin{aligned}
\left(\iint_{E_1}|f(k^{-1})|^p\mathrm{d}\sigma_k\right)^{\frac{1}{p}}&\leq 2^{\frac{p-1}{p}}\left(\iint_{E_1}|f_{(0)}(k)-f_{(0)}(k_1)|^p\mathrm{d}\sigma_k+ \pi|f_{(0)}(k_1)|^p\right)^{\frac{1}{p}}\\
&\leq 2^{\frac{p-1}{p}}\left(\|f_{(0)}(k)\|^p_{h_\alpha(E_1)}\iint_{E_1}|k-k_1|^{p\alpha}\mathrm{d}\sigma_k+ \pi\|f_{(0)}(k)\|_{L^\infty(E_1)}^p\right)^{\frac{1}{p}}\\
&\leq M(p,\alpha)\|f(k^{-1})\|_{C^\alpha(E_1)},
\end{aligned}\]
which implies that $f_{(0)}(k)\in L^p(E_1)$.
Here we have used the inequalities for $a,b>0$
\[\begin{aligned}
&(a+b)^\theta\leq 2^{\theta-1}(a^\theta+b^\theta), \quad &\theta\geq1;\\
&(a+b)^\theta\leq (a^\theta+b^\theta)\leq 2^{1-\theta}(a+b)^\theta, \quad & 0<\theta<1.
\end{aligned}\]
Thus, if $f(k)\in {C}^\alpha(\mathbb{C})$, then $f(k)\in{L}^{p,0}(\mathbb{C})$,  and
\begin{equation}\label{dbc23}
\|f(k)\|_{{L}^{p,0}(\mathbb{C})}\leq M(p,\alpha) \|f(k)\|_{{C}^\alpha(\mathbb{C})}.
\end{equation}
Moreover, we can get that ${C}^\alpha(\mathbb{C})\subset{L}^{p,\nu}(\mathbb{C})$, $p\nu<2$.

Since $C^\alpha(\mathbb{C})\subset L^{p,0}(\mathbb{C})$, we can choose
\[\psi(k;x)\in {L}_x^\infty(\mathbb{R},\mathbf{C}_k^{\alpha}(\mathbb{C}))\subset {L}_x^\infty(\mathbb{R},\mathbf{L}_k^{p,0}(\mathbb{C})),\]
then from the theorem \ref{dbth4} we obtain that if $r_\pm \in L_k^{q,2}(\mathbb{C})$ then
\[RT_\mathbb{C}: \mathbf{C}_k^{\alpha}(\mathbb{C})\to \mathbf{C}_k^\alpha(\mathbb{C}), \quad 0<\alpha<1.\]

It is remarked that the operator can also be obtained by directly showing the conditions
\[\|g^\pm(k;x)\|_{\mathbf{C}_k^{\alpha}(\mathbb{C})}\leq M_\pm^{p,q,\alpha}\|r_\mp\|_{L^{q,2}(\mathbb{C}^\pm)}\sup\limits_{x>0}\|\psi\|_{\mathbf{C}_k^{\alpha}(\mathbb{C}^\pm)},\]
and
\[\|h^\pm(k;x)\|_{\mathbf{C}_k^{\alpha}(\mathbb{C})}\leq N_\pm^{p,q,\alpha}\|r_\pm\|_{L^{q,2}(\mathbb{C}^\pm)}\sup\limits_{x>0}\|\psi\|_{\mathbf{C}_k^{\alpha}(\mathbb{C}^\pm)},\]
but the associated estimates will be more complicated.

\begin{theorem}\label{dbco1}
Let $\psi(k;x)\in {L}_x^\infty(\mathbb{R},\mathbf{C}_k^{\alpha}(\mathbb{C}))$, $r_\pm(k) \in L^{q,2}(\mathbb{C})$,
then $\psi RT_\mathbb{C}(k;x)\in \mathbf{C}_k^\alpha(\mathbb{C})$ and
\begin{equation}\label{dbc28}
\|RT_\mathbb{C}(k;x)\|_{\mathbf{C}_k^{\alpha}\to \mathbf{C}_k^\alpha}\leq M^{p,q,\alpha}(\|r_+\|_{L_k^{q,2}(\mathbb{C})}+\|r_-\|_{L_k^{q,2}(\mathbb{C})}).
\end{equation}
where $\frac{1}{p}+\frac{1}{q}<\frac{1}{2}$, and $\alpha=1-2\big(\frac{1}{p}+\frac{1}{q}\big)$.
\end{theorem}

\begin{theorem}
If $r_\pm \in L_k^{q,2}(\mathbb{C})$, then there exists $\psi(k;x)\in {L}_x^\infty(\mathbb{R},\mathbf{C}_k^{\alpha}(\mathbb{C}))$ satisfying the Dbar equation
\begin{equation}\label{dbknb1}
\bar\partial\psi(k;x)=\left\{\begin{array}{cc}
\psi(k;x)w^-(k;x)\chi_{\{x>0\}}+\psi(k;x)w^+(k;x)\chi_{\{x<0\}}, & k\in \mathbb{C}^+,\\
\psi(k;x)w^+(k;x)\chi_{\{x>0\}}+\psi(k;x)w^-(k;x)\chi_{\{x<0\}}, & k\in \mathbb{C}^-,
\end{array}\right.
\end{equation}
with normalization condition $\psi(k;x)\to I, k\to\infty$. In addition, the Dbar problem is equivalent to the following integral equation
\begin{equation}\label{dbc26}
\psi(k;x)=I+\psi RT_\mathbb{C}(k;x), \quad k\in \mathbb{C}.
\end{equation}
hold almost everywhere in $L_x^\infty(\mathbb{R},\mathbf{C}_k^\alpha(\mathbb{C}))$, where $\frac{1}{p}+\frac{1}{q}<\frac{1}{2}$, $\alpha=1-2\big(\frac{1}{p}+\frac{1}{q}\big)$ and $\psi RT_\mathbb{C}(k;x)$ is defined in \eqref{dbc5}.
\end{theorem}
Proof. It is verified that the Dbar equation \eqref{dbknb1} with normalization condition $\psi(k;x)\to I, n\to\infty$ gives to the integral equation \eqref{dbc26} by virtue of the generalized Cauchy integral formula or the socalled Pompeiu formula
\[\phi(k)=\frac{1}{2\pi i}\int_{C^\infty}\frac{\phi(z)}{z-k}dz+\frac{1}{2\pi i}\iint_\mathbb{C}\frac{\partial_{\bar{z}}\phi(z)}{z-k}dz\wedge d\bar{z},\]
where $C^\infty$ is a circle with an infinite radius.

Conversely, for $k\in \mathbb{C}^+$, since $\psi w_\pm T_{\mathbb{C}^-}(k;x)$ are holomorphic via the theorem \ref{dbth1}, then from the theorem \ref{dbth2}, we have
\[\begin{aligned}
\bar\partial(\psi RT_\mathbb{C}(k;x))&=\bar\partial(\psi w_-T_{\mathbb{C}^+}(k;x))\chi_{\{x>0\}}+\bar\partial(\psi w_+T_{\mathbb{C}^+}(k;x))\chi_{\{x<0\}}\\
&=\psi(k;x)w^-(k;x)\chi_{\{x>0\}}+\psi(k;x)w^+(k;x)\chi_{\{x<0\}}, \quad k\in \mathbb{C}^+.
\end{aligned}\]
Similarly, we have
\[\bar\partial(\psi RT_\mathbb{C}(k;x))=\psi(k;x)w^+(k;x)\chi_{\{x>0\}}+\psi(k;x)w^-(k;x)\chi_{\{x<0\}}, \quad k\in \mathbb{C}^-.\]
Thus, 
$\psi(k;x)$ given by \eqref{dbc26} is the solution of the Dbar equation \eqref{dbknb1}.
The theorem is proved. \qquad $\square$

For the positive constant $M^{p,q,\alpha}$ in \eqref{dbc28} related to the $p,q$ and $\alpha$, if let
\begin{equation}\label{dbc32}
M^{p,q,\alpha}<\frac{1}{\|r_+\|_{L_k^{q,2}(\mathbb{C})}+\|r_-\|_{L_k^{q,2}(\mathbb{C})}},
\end{equation}
then the inverse operator $(\mathcal{I}-RT_\mathbb{C})^{-1}$ on $L^\infty(\mathbb{R},\mathbf{C}_k^\alpha(\mathbb{C}))$ is exist.

\setcounter{equation}{0}
\section{Potential reconstruction of AKNS spectral problem}

Now we consider the Dbar problem relevant to the AKNS spectral problem,  that is the spectral transformation function $R(k;x)$ admits \eqref{gaa6}. Under the normalization condition, if $r_\pm \in L_k^{q,2}(\mathbb{C})$ admitting the condition \eqref{dbc32}, then $\psi(k;x)\in {L}_x^\infty(\mathbb{R},\mathbf{C}_k^{\alpha}(\mathbb{C}))$ can be written as
\begin{equation}\label{dbc27}
\psi(k;x)=I(\mathcal{I}-RT_\mathbb{C})^{-1}(k;x),
\end{equation}

To construct the AKNS spectral problem, we consider the following properties.

\begin{lemma}\label{dbl4}
Let $r_\pm\in L_k^{q,2}(\mathbb{C})$ and $\psi_{11},\psi_{22}\in L_x^\infty(\mathbb{R},C_k^{\alpha}(\mathbb{C}))$, where $\psi=(\psi_{mn})_{2\times2}$. Let
\[\langle\psi_{11} r_+\rangle_\pm=\frac{1}{2\pi i}\iint_{\mathbb{C}^\pm}\psi_{11}(k;x)r_+(k)\mathrm{e}^{-2ikx}\mathrm{d}k\wedge\mathrm{d}\bar{k}~\chi_{\{x\lessgtr0\}},\]
\[\langle\psi_{22} r_-\rangle_\pm=\frac{1}{2\pi i}\iint_{\mathbb{C}^\pm}\psi_{22}(k;x)r_-(k)\mathrm{e}^{2ikx}\mathrm{d}k\wedge\mathrm{d}\bar{k}~\chi_{\{x\gtrless0\}},\]
then
\begin{equation}\label{dbd8}
|\langle\psi_{11} r_+\rangle_\pm|\leq M_{\pm,12}^{p,q,\alpha}\sup\limits_{x\lessgtr0}\|\psi_{11}\|_{C_k^{\alpha}(\mathbb{C}^\pm)}\|r_+\|_{L_k^{q,2}(\mathbb{C}^\pm)},
\end{equation}
\begin{equation}\label{dbd9}
|\langle\psi_{22} r_-\rangle_\pm|\leq M_{\pm,13}^{p,q,\alpha}\sup\limits_{x\gtrless0}\|\psi_{22}\|_{C_k^{\alpha}(\mathbb{C}^\pm)}\|r_-\|_{L_k^{q,2}(\mathbb{C}^\pm)},
\end{equation}
where $\frac{1}{p}+\frac{1}{q}<\frac{1}{2}$, and $\alpha=1-2\big(\frac{1}{p}+\frac{1}{q}\big)$.

\end{lemma}
{\bf Proof}. We only give the proof of the condition \eqref{dbd8}. To this end, we give the following decomposition
\[\begin{aligned}
\langle\psi_{11} r_+\rangle_\pm&=\frac{1}{2\pi i}\iint_{E_1^\pm}\psi_{11}(k;x)r_+(k)\mathrm{e}^{-2ikx}\mathrm{d}k\wedge \mathrm{d}\bar{k}~\chi_{\{x\lessgtr0\}}\\
&\quad+\frac{1}{2\pi i}\iint_{E_2^\pm}\psi_{11}(k;x)r_+(k)\mathrm{e}^{-2ikx}\mathrm{d}k\wedge \mathrm{d}\bar{k}~\chi_{\{x\lessgtr0\}}\\
&=A_1^\pm+A_2^\pm.
\end{aligned}\]
Then using the H\"older inequality, we find that
\[\begin{aligned}
|A_1^\pm|&\leq\frac{1}{\pi}\iint_{E_1^\pm}\sup\limits_{x\lessgtr0}|\psi_{11}(k;x)||r_+(k)|\mathrm{d}\sigma_k\\
&\leq\frac{1}{\pi}\left(\iint_{E_1^\pm}(\sup\limits_{x\lessgtr0}|\psi_{11}(k;x)||r_+(k)|)^\mu\mathrm{d}\sigma_k\right)^{\frac{1}{\mu}}
\left(\iint_{E_1^\pm}\mathrm{d}\sigma_k\right)^{\frac{1}{\nu}}\\
&\leq\frac{1}{\pi}(\frac{\pi}{2})^{\frac{1}{\nu}}\sup\limits_{x\lessgtr0}\|\psi_{11}(k;x)\|_{L_k^p(E_1^\pm)}\|r_+\|_{L_k^q(E_1^\pm)},
\end{aligned}\]
where $\frac{1}{\mu}+\frac{1}{\nu}=1, \mu>2$ and $\frac{1}{\mu}=\frac{1}{p}+\frac{1}{q}$. In addition, for $A_2^\pm$, we have
\[\begin{aligned}
A_2^\pm&=\frac{1}{2\pi i}\iint_{E_2^\pm}\psi_{11}(k;x)r_+(k)\mathrm{e}^{-2ikx}\mathrm{d}k\wedge \mathrm{d}\bar{k}~\chi_{\{x\lessgtr0\}}\\
&=\frac{1}{2\pi i}\iint_{E_1^\mp}\psi_{11}(k^{-1};x)r_+(k^{-1})\mathrm{e}^{-2ik^{-1}x}k^{-2}\bar{k}^{-2}\mathrm{d}k\wedge \mathrm{d}\bar{k}~\chi_{\{x\lessgtr0\}}.
\end{aligned}\]
Similarly, we obtain that
\[\begin{aligned}
|A_2^\pm|&\leq\frac{1}{\pi}\iint_{E_1^\mp}\sup\limits_{x\lessgtr0}|\psi_{(0),11}(k;x)||r_{(2),+}(k)||k|^{-2}\mathrm{d}\sigma_k\\
&\leq\frac{1}{\pi}\left(\iint_{E_1^\mp}\big(\sup\limits_{x\lessgtr0}|\psi_{(0),11}(k;x)||r_{(2),+}(k)|\big)^\mu\mathrm{d}\sigma_k\right)^{\frac{1}{\nu}}
\left(\iint_{E_1^\mp}|k|^{-2\nu}\mathrm{d}\sigma_k\right)^{\frac{1}{\nu}},
\end{aligned}\]
where $\frac{1}{\mu}+\frac{1}{\nu}=1, \mu>2$, then $0<\nu<1$.
Since
\[\iint_{E_1^\mp}|k|^{-2\nu}\mathrm{d}\sigma_k=\pi\int_0^1r^{1-2\nu}\mathrm{d}r=\frac{\pi}{2-2\nu},\]
then
\[|A_2^\pm|\leq\frac{1}{2-2\nu}\sup\limits_{x\lessgtr0}\|\psi_{(0),11}(k;x)\|_{L_k^p(E_1^\mp)}\|r_{(2),+}\|_{L_k^q(E_1^\mp)}.\]
Thus, we have
\[\begin{aligned}
|\langle\psi_{11} r_+\rangle_\pm|&\leq M_{12}^{p,q}(\sup\limits_{x\lessgtr0}\|\psi_{11}(k;x)\|_{L_k^p(E_1^\pm)}\|r_+\|_{L_k^q(E_1^\pm)}\\
&\quad +\sup\limits_{x\lessgtr0}\|\psi_{(0),11}\|_{L_k^p(E_1^\mp)}\|r_{(2),+}\|_{L_k^q(E_1^\mp)})\\
&\leq M_{12}^{p,q}(\sup\limits_{x\lessgtr0}\|\psi_{11}\|_{L_k^p(E_1^\pm)}+\sup\limits_{x\lessgtr0}\|\psi_{(0),11}\|_{L_k^p(E_1^\mp)})\\
&\quad\times(\|r_+\|_{L_k^q(E_1^\pm)}+\|r_{(2),+}\|_{L_k^q(E_1^\mp)})),
\end{aligned}\]
which proves the condition \eqref{dbd8} by using the condition
\[\sup\limits_{x\lessgtr0}\|\psi_{11}\|_{L_k^{p,0}(\mathbb{C}^\pm)}\leq M_\pm(p,q,\alpha) \sup\limits_{x\lessgtr0}\|\psi_{11}\|_{C_k^{\alpha}(\mathbb{C}^\pm)}.\]
The condition \eqref{dbd9} can be proved similarly. \qquad $\square$

\begin{theorem}\label{dbth5}
If $\psi\in C_x^1(\mathbb{R})\cap {L}_x^\infty(\mathbb{R},\mathbf{C}_k^{\alpha}(\mathbb{C}))$, and $r_\pm \in L_k^{q,2}(\mathbb{C})$,
$\frac{1}{p}+\frac{1}{q}<\frac{1}{2}$, and $\alpha=1-2\big(\frac{1}{p}+\frac{1}{q}\big)$,
then AKNS spectral problem can be obtained
\begin{equation}\label{dbd1}
\partial_x\psi(k;x)=-ik[\sigma_3,\psi(k;x)]+Q(x)\psi(k;x),
\end{equation}
where the potential matrix $Q$ is given by the spectral transformation matrix $R$ and the eigenfunction $\psi$ as following
\begin{equation}\label{dbd2}
Q=-i[\sigma_3,\langle \psi R\rangle]:=\left(\begin{matrix}
0&u\\
v&0
\end{matrix}\right),
\end{equation}
with
\[\langle \psi R\rangle=\langle \psi w_-\rangle_++\langle \psi w_+\rangle_-+\langle \psi w_+\rangle_++\langle \psi w_-\rangle_-. \]
Here
\begin{equation}\label{dbd7}
\begin{aligned}
\langle \psi w_\mp\rangle_\pm=\frac{1}{2\pi i}\iint_{\mathbb{C}^\pm}\psi(k;x)w_\mp(k;x)\mathrm{d}k\wedge\mathrm{d}\bar{k}~\chi_{\{x>0\}},\\
\langle \psi w_\pm\rangle_\pm=\frac{1}{2\pi i}\iint_{\mathbb{C}^\pm}\psi(k;x)w_\pm(k;x)\mathrm{d}k\wedge\mathrm{d}\bar{k}~\chi_{\{x<0\}}.
\end{aligned}
\end{equation}
\end{theorem}
{\bf Proof}. If $\psi\in {L}_x^\infty(\mathbb{R},\mathbf{C}_k^{\alpha}(\mathbb{C}))$, and $r_\pm \in L_k^{q,2}(\mathbb{C})$, then from lemma \ref{dbl4}, we know that $k\psi w_\mp T_{\mathbb{C}^\pm}(k;x)\chi_{\{x>0\}}$ and $k\psi w_\pm T_{\mathbb{C}^\pm}(k;x)\chi_{\{x<0\}}$ are well-defined, in view of that
\begin{equation}\label{dbda7}
\begin{aligned}
k\psi w_\mp T_{\mathbb{C}^\pm}(k;x)\chi_{\{x>0\}}&=\frac{1}{2i\pi}\iint_{\mathbb{C}^\pm}\frac{z\psi(z;x)w_\mp(z;x)}{z-k}\mathrm{d}z\wedge\mathrm{d}\bar{z}~\chi_{\{x>0\}}\\
&=k[\psi w_\mp T_{\mathbb{C}^\pm}(k;x)\chi_{\{x>0\}}]+\langle \psi w_\mp\rangle_\pm,
\end{aligned}
\end{equation}
and
\begin{equation}\label{dbdb7}
k\psi w_\pm T_{\mathbb{C}^\pm}(k;x)\chi_{\{x<0\}}=k[\psi w_\pm T_{\mathbb{C}^\pm}(k;x)\chi_{\{x>0\}}]+\langle \psi w_\pm\rangle_\pm.
\end{equation}

Since $\psi\in C_x^1(\mathbb{R})$, then differentiating of equation \eqref{dbc26} with respect to $x$, and using \eqref{dbc27}, \eqref{gaa5}, we find that
\begin{equation}\label{dbd3}
\partial_x\psi=-ik\psi \sigma_3RT_\mathbb{C}(I-RT_\mathbb{C})^{-1}+ik\psi RT_\mathbb{C}\sigma_3(I-RT_\mathbb{C})^{-1}.
\end{equation}
For the first term on the right, since $RT_\mathbb{C}=\mathcal{I}-I(\mathcal{I}-RT_\mathbb{C})$, then \cite{ijmpb4-1003}
\begin{equation}\label{dbd5}
RT_\mathbb{C}(\mathcal{I}-RT_\mathbb{C})^{-1}=(\mathcal{I}-RT_\mathbb{C})^{-1}-\mathcal{I}.
\end{equation}
For the second term on the right of \eqref{dbd3}, we know, from \eqref{dbda7}, \eqref{dbdb7} and \eqref{dbc5}, that
\[\begin{aligned}
k\psi RT_\mathbb{C}=&[k\psi w_-T_{\mathbb{C}^+}(k;x)+k\psi w_+T_{\mathbb{C}^-}(k;x)]\chi_{\{x>0\}}\\
 &+[k\psi w_-T_{\mathbb{C}^-}(k;x)+k\psi w_+T_{\mathbb{C}^+}(k;x)]\chi_{\{x<0\}}\\
 &=k[\psi RT_\mathbb{C}]+\langle \psi R\rangle,
\end{aligned}\]
which can be reduced to
\begin{equation}\label{dbd4}
k\psi RT_\mathbb{C}=k\psi-kI+\langle \psi R\rangle,
\end{equation}
in terms of \eqref{dbc26}.

Substituting \eqref{dbd5} and \eqref{dbd4} into \eqref{dbd3}, and using \eqref{dbc27}, we can get that
\begin{equation}\label{dbd6}
\partial_x\psi=ik\psi \sigma_3+i\langle \psi R\rangle \sigma_3\psi-i\sigma_3k(I-RC_k)^{-1}.
\end{equation}
In addition, from \eqref{dbd4} and \eqref{dbc27}, we obtain that
\[k(I-RC_k)^{-1}=k\psi+\langle \psi R\rangle(I-RC_k)^{-1}=k\psi+\langle \psi R\rangle\psi.\]
It is remarked that $\langle \psi R\rangle$ is independent of $k$, so the last term can be obtained via \eqref{dbc27}.
The proof of the theorem is finished. \qquad $\square$

\begin{theorem}
Let $r_\pm\in L_k^{q,2}(\mathbb{C})$, $\frac{1}{p}+\frac{1}{q}<\frac{1}{2}$, and $\alpha=1-2\big(\frac{1}{p}+\frac{1}{q}\big)$,
then the potentials $u$ and $v$ defined in \eqref{dbd2} satisfy the conditions
\begin{equation}\label{dbd10}
\begin{aligned}
|u(x)|\leq \frac{M_{14}^{p,q}\|r_+\|_{L_k^{q,2}(\mathbb{C})}}{1-M_{16}^{p,q}(\|r_+\|_{L_k^{q,2}(\mathbb{C})}+\|r_-\|_{L_k^{q,2}(\mathbb{C})})},\\
|v(x)|\leq \frac{M_{15}^{p,q}\|r_-\|_{L_k^{q,2}(\mathbb{C})}}{1-M_{16}^{p,q}(\|r_+\|_{L_k^{q,2}(\mathbb{C})}+\|r_-\|_{L_k^{q,2}(\mathbb{C})})}.
\end{aligned}
\end{equation}
\end{theorem}
{\bf Proof}. From \eqref{dbd2} and \eqref{dbd7}, we know that
\[\begin{aligned}
u=\langle\psi_{11} r_+\rangle_+ + \langle\psi_{11} r_+\rangle_-,\\
v=\langle\psi_{22} r_-\rangle_+ + \langle\psi_{22} r_-\rangle_-.
\end{aligned}\]
Then using the conditions \eqref{dbd8} and \eqref{dbd9}, we can get that
\begin{equation}\label{dbd11}
\begin{aligned}
|u(x)|\leq M_{14}^{p,q}\sup\limits_{x\in\mathbb{R}}\|\psi_{11}\|_{C_k^{\alpha}(\mathbb{C})}\|r_+\|_{L_k^{q,2}(\mathbb{C})},\\
|v(x)|\leq M_{15}^{p,q}\sup\limits_{x\in\mathbb{R}}\|\psi_{22}\|_{C_k^{\alpha}(\mathbb{C})}\|r_-\|_{L_k^{q,2}(\mathbb{C})}.
\end{aligned}
\end{equation}

We have proved that the inverse operator $(I-RT_\mathbb{C})^{-1}$ is exist based on the small norm condition. According to the identity \eqref{dbd5}, the inverse operator admits the relation
\[(\mathcal{I}-RT_\mathbb{C})^{-1}=\mathcal{I}+RT_\mathbb{C}(\mathcal{I}-RT_\mathbb{C})^{-1},\]
which maps the space $\mathbf{C}_k^\alpha(\mathbb{C})$ into itself. Thus, we find that
\[\begin{aligned}
\|(\mathcal{I}-RT_\mathbb{C})^{-1}\|_{\mathbf{C}_k^\alpha(\mathbb{C})\to \mathbf{C}_k^{\alpha}(\mathbb{C})}&\leq\frac{1}{1-\|RT_\mathbb{C}\|_{\mathbf{C}_k^{\alpha}(\mathbb{C})\to \mathbf{C}_k^\alpha(\mathbb{C})}}\\
&\leq
\frac{1}{1-M_{16}^{p,q}(\|r_+\|_{L_k^{q,2}(\mathbb{C})}+\|r_-\|_{L_k^{q,2}(\mathbb{C})})},
\end{aligned}\]
in terms of the condition \eqref{dbc28}.
In addition, the representation of $\psi(k;x)$ is given via the inverse operator in \eqref{dbc27}. Since $I\in C^\alpha(\mathbb{C})$, then we have
\begin{equation}\label{dbd12}
\begin{aligned}
\sup\limits_{x\in\mathbb{R}}\|\psi(k;x)\|_{\mathbf{C}_k^{\alpha}(\mathbb{C})}\leq \frac{1}{1-M_{16}^{p,q}(\|r_+\|_{L_k^{q,2}(\mathbb{C})}+\|r_-\|_{L_k^{q,2}(\mathbb{C})})}.
\end{aligned}
\end{equation}
Hence, the conditions \eqref{dbd10} can be proved via \eqref{dbd11} and \eqref{dbd12}. \qquad $\square$

\begin{corollary}
The map
\[L^{q,2}(\mathbb{C})\ni r_\pm(k)\to u(x),v(x)\in L^\infty(\mathbb{R})\]
is Lipschitz continuous.
\end{corollary}
{\bf Proof}. Let $\tilde{r}_\pm(k)$ and $\tilde\psi(k;x)=(\tilde\psi_{mn})$ are relevant to the potential $\tilde{u}(x), \tilde{v}(x)$ admitting the relation \eqref{dbd2}, and $\tilde{r}_\pm(k)$,  $\tilde\psi(k;x)=(\tilde\psi_{mn})$ satisfy the relation \eqref{dbc27}, then
\begin{equation}\label{dbd15}
\begin{aligned}
u-\tilde{u}=&\langle\psi_{11}(r_+-\tilde{r}_+)\rangle_+ + \langle(\psi_{11}-\tilde\psi_{11})\tilde{r}_+\rangle_+\\
&+\langle\psi_{11}(r_+-\tilde{r}_+)\rangle_- + \langle(\psi_{11}-\tilde\psi_{11})\tilde{r}_+\rangle_-.
\end{aligned}
\end{equation}
Using \eqref{dbd8} and \eqref{dbd9}, we can get
\[\begin{aligned}
|u-\tilde{u}|\leq M_{14}^{p,q}\big[\sup\limits_{x\in\mathbb{R}}\|\psi_{11}\|_{C_k^{\alpha}(\mathbb{C})}\|r_+-\tilde{r}_+\|_{L_k^{q,2}(\mathbb{C})}\\
+\sup\limits_{x\in\mathbb{R}}\|\psi_{11}-\tilde\psi_{11}\|_{C_k^{\alpha}(\mathbb{C})}\|\tilde{r}_+\|_{L_k^{q,2}(\mathbb{C})}\big].
\end{aligned}\]

From the relation \eqref{dbc27}, we find that
\[\begin{aligned}
\psi-\tilde\psi&=I(\mathcal{I}-RT_\mathbb{C})^{-1}-I(\mathcal{I}-\tilde{R}T_\mathbb{C})^{-1}\\
&=I(\mathcal{I}-RT_\mathbb{C})^{-1}~(R-\tilde{R})T_\mathbb{C}~(\mathcal{I}-\tilde{R}T_\mathbb{C})^{-1}.
\end{aligned}\]
In a similar way, we have
\[\begin{aligned}
&\sup\limits_{x\in\mathbb{R}}\|\psi-\tilde\psi\|_{\mathbf{C}_k^{\alpha}(\mathbb{C})}\\
&\leq\frac{M^{p,q}(\|r_+-\tilde{r}_+\|_{L_k^{q,2}}+\|r_--\tilde{r}_-\|_{L_k^{q,2}})}{[1-M_{16}^{p,q}(\|r_+\|_{L_k^{q,2}(\mathbb{C})}+\|r_-\|_{L_k^{q,2}(\mathbb{C})})][1-\tilde{M}_{16}^{p,q}(\|\tilde{r}_+\|_{L_k^{q,2}(\mathbb{C})}+\|\tilde{r}_-\|_{L_k^{q,2}(\mathbb{C})})]}.
\end{aligned} \]

If let $\|r_\pm\|_{L_k^{q,2}},\|\tilde{r}_\pm\|_{L_k^{q,2}}<B$, then, from \eqref{dbd15}, we can obtain
\begin{equation}\label{dbd16}
|u-\tilde{u}|\leq M(p,q,B)\big(\|r_+-\tilde{r}_+\|_{L_k^{q,2}}+\|r_--\tilde{r}_-\|_{L_k^{q,2}}\big).
\end{equation}
Similar discussion can derive the same condition as \eqref{dbd16} for the potential $v$. \qquad $\square$

\setcounter{equation}{0}
\section{Conclusions and Discussions}
We presented properties for the integral operator $\hat{T}_G$ associated with the Dbar equation $\bar\partial\phi(k)=f(k)$ on the bounded domain $G$, and showed that $\hat{T}_G$ was a completely continuous operator on the space $L^p(\bar{G}), p>2$ and mapped it into the space $C^\gamma(\bar{G}), \gamma=\frac{p-2}{p}$. Since the relevant constants were dependent on the radius of the domain, so it was difficult to let the operator was small norm, which implied that the condition for the exist of the inverse operator $(I-\hat{T}_G)^{-1}$ was difficult to achieve.

To apply the Dbar problem to integrable system, one needed to
introduce the spectral transformation off-diagonal matrix $R$ admitting the linear evolution equation about the physical variable $x$ as $\partial_xR(k;x)=-ik[\sigma_3,R(k;x)]$.
To control the convergence of the relevant integral, we split the spectral transformation function $R$ into two nilpotent matrices $w_\pm(k;x)$ containing the factors $r_\pm(k)\mathrm{e}^{\mp2ikx}$, and further split the complex plane $\mathbb{C}$ into lower and upper half planes $\mathbb{C}^\pm$. So the integral $\psi R{T}_\mathbb{C}$ was decomposed into the sum of the factors $\psi w_\mp T_{\mathbb{C}^\pm}(k;x)\chi_{\{x>0\}}$ and $\psi w_\pm T_{\mathbb{C}^\pm}(k;x) \chi_{\{x<0\}}$.

We further split the lower and upper half planes $\mathbb{C}^\pm$ into lower/upper inner half unit domain $E_1^\pm$ and lower/upper outer half unit domain $E_2^\pm$. By introducing the space $L^{p,\nu}(\mathbb{C}^\pm)$, we gave the prior estimates for the decomposed integrals $\psi w_\mp T_{\mathbb{C}^\pm}(k;x)\chi_{\{x>0\}}$ and $\psi w_\pm T_{\mathbb{C}^\pm}(k;x) \chi_{\{x<0\}}$, and showed that if $\psi\in {L}_x^\infty(\mathbb{R},\mathbf{L}_k^{p,0}(\mathbb{C}))$, $r_\pm \in L_k^{q,2}(\mathbb{C})$, then $\psi RT_\mathbb{C}(k;x)\in L_x^\infty(\mathbb{R},\mathbf{C}_k^\alpha(\mathbb{C}))$. Using the fact $\mathbf{C}_k^\alpha(\mathbb{C}) \subset L_k^{q,2}(\mathbb{C})$ to choose $\psi\in {L}_x^\infty(\mathbb{R},\mathbf{C}_k^{\alpha}(\mathbb{C}))$, we found that small norm condition for the operator $RT_\mathbb{C}$, and showed that the inverse operator $(\mathcal{I}-RT_\mathbb{C})^{-1}$ was exist.

In this paper, we consider the no soliton case, that is no Dirac delta factors in the spectral transformation $R$. We showed that, if $r_\pm \in L_k^{q,2}(\mathbb{C})$, there exists a unique solution $\psi=I(\mathcal{I}-RT_\mathbb{C})^{-1}\in {L}_x^\infty(\mathbb{R},\mathbf{C}_k^{\alpha}(\mathbb{C}))$ for the Dbar problem.
By using the Dbar dressing method, we constructed the AKNS spectral problem $\partial_x\psi=-ik[\sigma_3,\psi]+Q\psi$,
and established the potential reconstruction $Q=-i[\sigma_3,\langle \psi R\rangle]$. By decomposing $\langle \psi R\rangle$ into $\langle \psi w_\mp\rangle_\pm$ and $\langle \psi w_\pm\rangle_\pm$, we showed that the map from $r_\pm\in L_k^{q,2}(\mathbb{C})$ to $u,v\in L_x^\infty(\mathbb{R})$ is Lipschitz continuous.

It is remarked that, in the direct scattering analysis for the AKNS spectral problem, the potentials $u,v$ are usually be taken as $L^2(\mathbb{R})$ functions. So, for the this condition, we find from \eqref{dbd11} that
\begin{equation}\label{dbe1}
\begin{aligned}
\|u(x)\|_{L^2(\mathbb{R})}\leq M_{14}^{p,q}\|\psi_{11}\|_{L_x^2(\mathbb{R},C_k^{\alpha}(\mathbb{C}))}\|r_+\|_{L_k^{q,2}(\mathbb{C})},\\
\|v(x)\|_{L^2(\mathbb{R})}\leq M_{15}^{p,q}\|\psi_{22}\|_{L_x^2(\mathbb{R},C_k^{\alpha}(\mathbb{C}))}\|r_-\|_{L_k^{q,2}(\mathbb{C})}.
\end{aligned}
\end{equation}
Moreover, if $\psi\in L_x^2(\mathbb{R},\mathbf{C}_k^{\alpha}(\mathbb{C}))$, then we can get, from \eqref{dbc21} and \eqref{dbc22}, that
\begin{equation}\label{dbe2}
\|\psi RT_\mathbb{C}(k;x)\|_{L_x^2(\mathbb{R})}\leq M_{10}^{p,q}(\|r_+\|_{L_k^{q,2}(\mathbb{C})}+\|r_-\|_{L_k^{q,2}(\mathbb{C})})\|\psi\|_{L_x^2(\mathbb{R},C_k^{\alpha}(\mathbb{C}))},
\end{equation}
and for $k_1,k_2\in\mathbb{C}$,
\begin{equation}\label{dbe3}
\begin{aligned}
&\|\psi RT_\mathbb{C}(k_1;x)-\psi RT_\mathbb{C}(k_2;x)\|_{L_x^2(\mathbb{R})}\\
&\leq M_{11}^{p,q}(\|r_+\|_{L_k^{q,2}(\mathbb{C})}+\|r_-\|_{L_k^{q,2}(\mathbb{C})})\|\psi\|_{L_x^2(\mathbb{R},C_k^{\alpha}(\mathbb{C}))}\|k_1-k_2|^\alpha,
\end{aligned}
\end{equation}
where $\frac{1}{p}+\frac{1}{q}<\frac{1}{2}$, and $\alpha=1-2\big(\frac{1}{p}+\frac{1}{q}\big)$. So the conditions \eqref{dbe2} and \eqref{dbe3} imply that the operator $RT_\mathbb{C}$ is also a map from space $L_x^2(\mathbb{R},\mathbf{C}_k^{\alpha}(\mathbb{C}))$ to $L_x^2(\mathbb{R},\mathbf{C}_k^\alpha(\mathbb{C}))$, and
\begin{equation}\label{dbe4}
\|RT_\mathbb{C}(k;x)\|_{L_x^2(\mathbb{R},\mathbf{C}_k^{\alpha}(\mathbb{C}))\to L_x^2(\mathbb{R},\mathbf{C}_k^\alpha(\mathbb{C}))}\leq M^{p,q}(\|r_+\|_{L_k^{q,2}(\mathbb{C})}+\|r_-\|_{L_k^{q,2}(\mathbb{C})}).
\end{equation}
Under the condition of small norm, the inverse operator $(\mathcal{I}-RT_\mathbb{C})^{-1}$ is also exist, and it can be used to define the eigenfunction $\psi$ in \eqref{dbc27}. We similar have
\[\begin{aligned}
\|u(x)\|_{L^2(\mathbb{R})}\leq \frac{M_{14}^{p,q}\|r_+\|_{L_k^{q,2}(\mathbb{C})}}{1-M_{16}^{p,q}(\|r_+\|_{L_k^{q,2}(\mathbb{C})}+\|r_-\|_{L_k^{q,2}(\mathbb{C})})},\\
\|v(x)\|_{L^2(\mathbb{R})}\leq \frac{M_{15}^{p,q}\|r_-\|_{L_k^{q,2}(\mathbb{C})}}{1-M_{16}^{p,q}(\|r_+\|_{L_k^{q,2}(\mathbb{C})}+\|r_-\|_{L_k^{q,2}(\mathbb{C})})},
\end{aligned}\]
which imply that the map
\[L^{q,2}(\mathbb{C})\ni r_\pm(k) \to u(x),v(x)\in L^2(\mathbb{R}),\]
is Lipschitz continuous.

It is also remarked that, the paper only considered the well-posedness for the AKNS spectral problem. If consider the Dbar problem associated with a nonlinear integrable equation, one needs to introduce a time evolution equation for the spectral transformation function $R$, for example $\partial_tR=A(k)[\sigma_3,R]$, where $A(k)$ is a polynomial.
So the exponents will take the form $e^{\pm 2i\theta(k;x,t)}$. To ensure the convergence of the relevant integral, one needs to give a more complicated decomposition about the domain.
These cases will be considered in our subsequent articles.

\section*{Statements and Declarations}
\noindent{\bf Authorship contribution statements} Junyi Zhu: Conceptualization, Investigation, Methodology, Validation, Writing-review and editing; Huan Liu: Investigation, Methodology, Validation, Visualization, Project administration.

\noindent{\bf Conflict of interest} The authors declare that they have no conflicts of interest.

\noindent{\bf Fundings}
This work was supported by National Natural Science Foundation of China (Grant No. 12571268), Natural Science Foundation of Henan Province (Grant No. 262300421239), and Program for Science \& Technology Innovation Talents in Universities of Henan Province (Grant No. 26HASTIT049).


\end{document}